\documentclass[reqno]{amsart}
\usepackage[a4paper,margin=32mm]{geometry}
\usepackage[T1]{fontenc}
\usepackage[utf8]{inputenc}
\usepackage{amsmath,amssymb,amsthm,mathtools}
\usepackage{enumitem}
\usepackage{url}
\usepackage[hidelinks]{hyperref}

\theoremstyle{plain}
\newtheorem{theorem}{Theorem}[section]
\newtheorem{lemma}[theorem]{Lemma}
\newtheorem{proposition}[theorem]{Proposition}
\newtheorem{corollary}[theorem]{Corollary}

\theoremstyle{definition}
\newtheorem{definition}[theorem]{Definition}

\theoremstyle{remark}
\newtheorem{remark}[theorem]{Remark}

\newcommand{\AM}{\mathrm{AM}}
\newcommand{\cF}{\mathcal{F}}
\newcommand{\N}{\mathbb{N}}
\newcommand{\czerozero}{c_{00}}

\newcommand{\hatot}{\widehat{\otimes}}

\renewcommand{\le}{\leqslant}
\renewcommand{\ge}{\geqslant}
\newcommand{\supp}{\operatorname{supp}}
\newcommand{\ad}{\mathrm{ad}}
\newcommand{\WAM}{\mathrm{WAM}}  

\newcommand{\Zc}{\mathrm{Z}} 

\title[Amenability of unconditional sums]{Amenability constants for unconditional sums\\ of Banach algebras}

\author[T.~Kania]{Tomasz Kania}
\address[T.~Kania]{Mathematical Institute\\Czech Academy of Sciences\\\v{Z}itn\'a 25 \\115 67 Praha 1\\Czech Republic and Institute of Mathematics and Computer Science\\ Jagiellonian University\\ {\L}ojasiewicza 6, 30-348 Krak\'{o}w, Poland}
\email{kania@math.cas.cz, tomasz.marcin.kania@gmail.com}

\thanks{RVO: 67985840.}

\author{Jerzy K\c akol}
\address[J.~K\c akol]{ Faculty of Mathematics and Informatics, A. Mickiewicz University,
61-614 Pozna\'{n}, Poland}
\email{kakol@amu.edu.pl}

\date{\today}

\thanks{Tomasz Kania is the corresponding author.}
\begin{document}
\begin{abstract}
We study Johnson amenability for unconditional direct sums of Banach algebras.
Given a family $(A_i)_{i\in I}$ of Banach algebras and a Banach sequence lattice $E$ on~$I$,
the $E$-sum $\bigl(\bigoplus_{i\in I} A_i\bigr)_{\!E}$ carries a natural Banach algebra structure
via coordinatewise multiplication.
Under the hypothesis that $C_E := \sup\{\|\chi_F\|_E: F\subseteq I\text{ finite}\}<\infty$,
we prove that this $E$-sum is amenable if and only if the amenability constants of the summands
are uniformly bounded, and we establish the two-sided estimate
\[
\sup_{i\in I}\AM(A_i) \;\le\; \AM\Bigl(\bigl(\textstyle\bigoplus_{i\in I} A_i\bigr)_{\!E}\Bigr)
\;\le\; C_E^2\,\sup_{i\in I}\AM(A_i).
\]
We show that the factor $C_E^2$ is sharp by exhibiting finite-dimensional examples where equality holds.
We further prove that finiteness of $C_E$ is necessary whenever infinitely many summands are non-zero
and the sum admits a bounded approximate identity.

As applications, we recover the classical formula
$\AM\bigl(c_0\text{-}\bigoplus_{i\in I} A_i\bigr) = \sup_{i\in I}\AM(A_i)$
for arbitrary (possibly uncountable) index sets, extend it to weighted $c_0$-spaces,
and characterise amenability for Orlicz sequence algebra sums.
We also record how these unconditional criteria give obstructions within the
conditional framework of James-type $J$-sums.

Finally, we investigate weak amenability of $E$-sums.
We prove that weak amenability passes to summands, that $E$-sums of commutative
weakly amenable algebras are weakly amenable, and---contrasting sharply with the
Johnson amenability picture---that for $1 < p < \infty$, the $\ell_p$-sum of
infinitely many copies of a non-commutative weakly amenable algebra fails to be
weakly amenable.
In the $c_0$-type regime ($C_E < \infty$), we establish two-sided estimates
for weak amenability constants with constants depending only on $C_E$.
\end{abstract}
\maketitle

\section{Introduction}

Unconditional direct sums provide one of the most natural mechanisms for constructing
large Banach spaces and Banach algebras from smaller building blocks.
If $(A_i)_{i\in I}$ is a family of Banach algebras and $E$ is a Banach sequence lattice on~$I$,
the $E$-sum $\bigl(\bigoplus_{i\in I} A_i\bigr)_{\!E}$ inherits a canonical Banach algebra structure
through coordinatewise multiplication.
A fundamental question is then: which homological properties of the summands transfer to the sum,
and under what conditions?

The notion of amenability for Banach algebras, introduced by Johnson in his seminal
memoir~\cite{Johnson1972Cohomology}, has become one of the central themes in the
homological theory of Banach algebras.
Johnson's celebrated theorem characterising amenable group algebras $L^1(G)$ as precisely
those arising from amenable locally compact groups~$G$ established a deep link between
abstract harmonic analysis and Banach algebra cohomology.
Subsequent decades witnessed remarkable developments: Connes' characterisation of
injective von Neumann algebras~\cite{Connes1976}, Haagerup's proof that the reduced
$C^*$-algebra of a free group is not amenable~\cite{Haagerup1983},
and the systematic study of amenability constants initiated by
Johnson~\cite{Johnson1994} and pursued by many authors;
see the monographs~\cite{Helemskii1989,Dales2000,Runde2002Lectures,Runde2020ABA}
for comprehensive treatments and further references.
More recently, numerous variants such as approximate amenability~\cite{GhahramaniLoy2004},
character amenability~\cite{SanganiMonfared2008}, and module amenability~\cite{Amini2004} have
been introduced and actively studied.

Johnson amenability~\cite{Johnson1972Cohomology} is a particularly robust homological invariant,
admitting an elegant characterisation via virtual diagonals in the bidual of the projective tensor product.
For the $c_0$-sum of a \emph{sequence} of Banach algebras, a folklore result
(recorded in Runde's monograph~\cite{Runde2020ABA}) asserts that the amenability constant of the sum
equals the supremum of the amenability constants of the summands.
Our first objective is to provide a self-contained proof of this formula that applies verbatim
to arbitrary index sets---including uncountable ones---and, more generally,
to any Banach sequence lattice~$E$ satisfying the uniform boundedness condition
\[
C_E := \sup\bigl\{\|\chi_F\|_E : F \subseteq I \text{ finite}\bigr\} < \infty.
\]
This condition captures the requirement that $E$ behaves uniformly like $c_0$ on finite supports.
The extremal case $C_E=1$ occurs precisely when $E=c_0(I)$ with its usual supremum norm,
while the more general hypothesis $C_E<\infty$ allows weighted variants and other (equivalent)
lattice renormings of $c_0(I)$.

Our main result (Theorem~\ref{thm:main}) establishes that, under the hypothesis $C_E < \infty$,
the $E$-sum is amenable if and only if $\sup_{i \in I} \AM(A_i) < \infty$,
with the quantitative estimate
\[
\sup_{i\in I}\AM(A_i) \;\le\; \AM(A) \;\le\; C_E^2\,\sup_{i\in I}\AM(A_i).
\]
We prove that the constant $C_E^2$ cannot, in general, be replaced by $C_E$:
for the finite $\ell_2^n$-sum of $n$ copies of $\mathbb{C}$, one has $\AM(A) = n = C_E^2$
(Proposition~\ref{prop:CE2-sharp}).
Conversely, we show that amenability of an $E$-sum with infinitely many non-zero summands
forces $C_E < \infty$ (Proposition~\ref{prop:CE-necessary}), so the finiteness hypothesis is
essentially optimal.

A second goal of this paper is conceptual.
Certain constructions in Banach space theory are \emph{conditional} rather than unconditional;
the prototypical example is Bellenot's James-type $J$-sum~\cite{Bellenot1982},
which underpins the Lindenstrauss construction of spaces $Y$ with $Y^{**}/Y$ isomorphic to a prescribed separable space. Even in such conditional settings, unconditional $E$-sums frequently appear as coordinate algebras or natural subalgebras.

We therefore record how the unconditional criteria give quotient obstructions inside larger
conditional structures, such as Bellenot $J$-sums.  We do not use, or assert, the existence
of non-trivial amenable $J$-sum algebras; the $J$-sum discussion is formulated as a source
of necessary conditions and non-amenability tests (Corollary~\ref{cor:Esum-detected}).

We highlight explicit consequences for Banach-algebra-valued sequence algebras.
When the family $(A_i)_{i \in I}$ is constant, \emph{i.e.}, $A_i = B$ for all $i$, the $E$-sum $B^{(E)}$ may be viewed as a tensor-product-type construction (Remark~\ref{rem:tensor-viewpoint}).
Corollary~\ref{cor:Piszczek} then provides a Banach-algebraic counterpart of results
in the spirit of Piszczek's work on amenability of K\"othe co-echelon algebras~\cite{Piszczek2019}.

A third goal is to investigate the extent to which these results have analogues for
\emph{weak amenability}: the property that every bounded derivation into the dual module is inner.
We prove that weak amenability passes from an $E$-sum to its coordinate algebras
(Proposition~\ref{prop:WA-to-summands}), and that $E$-sums of commutative weakly amenable
algebras are weakly amenable (Theorem~\ref{thm:WA-counterpart}(2)).
However, the picture for non-commutative summands is strikingly different from Johnson ame\-na\-bi\-li\-ty:
for $1 < p < \infty$, the $\ell_p$-sum of infinitely many copies of a non-commutative weakly amenable
algebra \emph{fails} to be weakly amenable (Theorem~\ref{thm:WA-counterpart}(3)),
a phenomenon first observed by Koczorowski and Piszczek~\cite{KoczorowskiPiszczek2024WeakAmenability}
for constant families.

In the $c_0$-type regime ($C_E < \infty$), we establish a quantitative counterpart:
weak amenability of the $E$-sum is equivalent to uniform weak amenability of the summands,
with two-sided estimates for the weak amenability constant whose constants depend only on $C_E$
(Corollary~\ref{cor:WAM-CE}).

\medskip
\noindent\textbf{Organisation.}
Section~\ref{sec:amenability-constants} recalls virtual diagonals, amenability constants,
and two standard permanence results (quotients and dense unions).
Section~\ref{sec:E-sums} introduces Banach sequence lattices and the associated $E$-sums.
The main theorem on Johnson amenability is proved in Section~\ref{sec:main}.
Section~\ref{sec:applications} presents applications: $c_0$-sums over arbitrary index sets,
weighted $c_0$-spaces, Orlicz sequence spaces, the sharpness of the constant $C_E^2$,
and constant families of summands.
Section~\ref{sec:Jsums} discusses $J$-sums and explains how unconditional quotient
structures yield obstructions to amenability of conditional constructions.
Section~\ref{sec:weak-amenability} develops the weak amenability theory for $E$-sums,
highlighting both the parallels with and the divergences from Johnson amenability.

\medskip
\noindent\textbf{Notation.}
Throughout, $I$ denotes a non-empty set and $\mathcal{F}(I)$ the collection of its finite subsets.
We write $c_{00}(I)$ for the space of finitely supported complex-valued functions on $I$,
$\chi_F$ for the characteristic function of $F \subseteq I$,
and $\delta_i$ for the Dirac mass at $i \in I$.
For a Banach algebra $A$, we denote by $A \hatot A$ the projective tensor product
and by $\pi\colon A \hatot A \to A$ the multiplication map.

\section{Amenability constants and dense unions}\label{sec:amenability-constants}

We follow the standard formulation of Johnson amenability via virtual diagonals; see~\cite{Johnson1972Cohomology,Runde2002Lectures}.

\begin{definition}[Virtual diagonal and amenability constant]\label{def:AM}
Let $A$ be a Banach algebra and let $\pi\colon A\hatot A\to A$ denote the product map $\pi(a\otimes b)=ab$.
We view $A\hatot A$ as an $A$-bimodule via
\[
a\cdot(x\otimes y)=(ax)\otimes y, \qquad (x\otimes y)\cdot a=x\otimes(ya),
\]
and extend these actions canonically to $(A\hatot A)^{**}$.

An element $M\in (A\hatot A)^{**}$ is called a \emph{virtual diagonal} for~$A$ if
\[
a\cdot M = M\cdot a \quad\text{and}\quad a\cdot \pi^{**}(M)=a=\pi^{**}(M)\cdot a
\qquad (a\in A).
\]
The algebra $A$ is \emph{amenable} if it admits a virtual diagonal.

If $A$ is amenable, its \emph{amenability constant} is
\[
\AM(A):=\inf\bigl\{\|M\|\colon M \text{ is a virtual diagonal for }A \bigr\},
\]
and we set $\AM(A)=+\infty$ if $A$ is not amenable.
\end{definition}

\begin{remark}\label{rem:AM-attained}
If $A$ is amenable, then the infimum in Definition~\ref{def:AM} is attained, i.e.\
there exists a virtual diagonal $M$ with $\|M\|=\AM(A)$.
Indeed, choose a net $(M_\alpha)_\alpha$ of virtual diagonals such that
$\|M_\alpha\|\downarrow \AM(A)$. This net is bounded, so by Banach--Alaoglu
it has a weak-$*$ cluster point $M$ in $(A\hat\otimes A)^{**}$.
The defining conditions for a virtual diagonal are weak-$*$ closed, hence $M$
is again a virtual diagonal. Since the norm on a dual space is weak-$*$ lower
semicontinuous, we have $\|M\|\le \liminf_\alpha \|M_\alpha\|=\AM(A)$, and therefore
$\|M\|=\AM(A)$. Compare \cite[Definition~2.2.7]{Runde2020ABA}.
\end{remark}

The following two lemmas are standard; we include short proofs for completeness.

\begin{lemma}\label{lem:quotient}
Let $q\colon A\to B$ be a surjective continuous homomorphism of Banach algebras. Then
\[
\AM(B)\le \|q\|^2\,\AM(A).
\]
\end{lemma}

\begin{proof}
The map $q\hatot q\colon A\hatot A\to B\hatot B$ is a surjective continuous homomorphism with
$\|q\hatot q\|\le \|q\|^2$, hence its bidual $(q\hatot q)^{**}$ has the same norm bound.
If $M$ is a virtual diagonal for~$A$, then $N:=(q\hatot q)^{**}(M)$ is a virtual diagonal for~$B$:
the defining identities follow immediately from functoriality of module actions and the relation
$\pi_B\circ (q\hatot q)=q\circ \pi_A$.
Therefore $\AM(B)\le \|N\|\le \|q\|^2\|M\|$, and taking the infimum over $M$ gives the claim.
\end{proof}

\begin{lemma}[Directed dense unions]\label{lem:dense-union}
Let $A$ be a Banach algebra and let $(B_\lambda)_{\lambda\in\Lambda}$ be an increasing net of closed subalgebras,
indexed by a directed set $(\Lambda,\le)$, such that $\bigcup_{\lambda\in\Lambda}B_\lambda$ is dense in~$A$.
Then
\[
\AM(A)\le \sup_{\lambda\in\Lambda}\AM(B_\lambda).
\]
\end{lemma}

\begin{proof}
This is the virtual-diagonal formulation of \cite[Proposition~2.3.15]{Runde2020ABA}.

If the right-hand side is $+\infty$ there is nothing to prove. Otherwise, let $C<\infty$ dominate
$\AM(B_\lambda)$ for all~$\lambda$.
Choose, for each~$\lambda$, a virtual diagonal $M_\lambda$ for~$B_\lambda$ with $\|M_\lambda\|=\AM(B_\lambda)$.
Let $i_\lambda\colon B_\lambda\hookrightarrow A$ denote the inclusion homomorphism, and set
\[
\widetilde M_\lambda := (i_\lambda\hatot i_\lambda)^{**}(M_\lambda)\in (A\hatot A)^{**}.
\]
Then $\|\widetilde M_\lambda\|\le \|M_\lambda\|\le C$, so $(\widetilde M_\lambda)_{\lambda\in\Lambda}$
is a bounded net in $(A\hatot A)^{**}$.
By the Banach--Alaoglu theorem, it has a weak-$*$ cluster point~$M$ with $\|M\|\le C$.

Fix $b\in\bigcup_{\lambda\in\Lambda}B_\lambda$ and choose $\lambda_0$ with $b\in B_{\lambda_0}$.
Then $b\in B_\lambda$ for all $\lambda\ge \lambda_0$, so $b\cdot M_\lambda=M_\lambda\cdot b$ and
$b\cdot\pi^{**}(M_\lambda)=b=\pi^{**}(M_\lambda)\cdot b$ for all $\lambda\ge\lambda_0$.
Passing to the weak-$*$ limit yields the same identities with $M$ in place of $M_\lambda$.

A standard density argument (approximating an arbitrary $a\in A$ by elements of $\bigcup_\lambda B_\lambda$)
shows that $M$ is a virtual diagonal for all of $A$. Hence $\AM(A)\le \|M\|\le C$.
\end{proof}

We will also need the following well-known finite-sum formula.

\begin{lemma}\label{lem:finite-infty}
Let $A_1,\dots,A_n$ be Banach algebras and let $A:=A_1\oplus_\infty\cdots\oplus_\infty A_n$
with coordinatewise multiplication and norm $\|(a_1,\dots,a_n)\|=\max_j\|a_j\|$.
Then
\[
\AM(A)=\max_{1\le j\le n}\AM(A_j).
\]
\end{lemma}

\begin{proof}
This is \cite[Lemma~2.3.18]{Runde2020ABA}.  (Runde proves the result using approximate diagonals.)
\end{proof}

\section{Unconditional $E$-sums of Banach algebras}\label{sec:E-sums}

We now specify the class of coefficient spaces~$E$ used to define unconditional sums.

\begin{definition}\label{def:sequence-lattice}
Let $I$ be a set. A \emph{Banach sequence lattice} on $I$ is a Banach space $E$
of complex-valued functions on $I$ such that:
\begin{enumerate}[label={\rm(\alph*)}]
\item $c_{00}(I)\subseteq E\subseteq \ell_\infty(I)$ and $c_{00}(I)$ is dense in $E$;
\item $E$ is \emph{solid}: if $y\in E$ and $x\in \ell_\infty(I)$ satisfy
$|x_i|\le |y_i|$ for all $i\in I$, then $x\in E$ and $\|x\|_E\le \|y\|_E$;
\item the inclusion $E\hookrightarrow \ell_\infty(I)$ is contractive, i.e.\ $\|x\|_\infty\le \|x\|_E$
for all $x\in E$.
\end{enumerate}
For each finite set $F\subseteq I$ let $\chi_F:=\sum_{i\in F}\delta_i\in c_{00}(I)\subseteq E$ and define
\[
C_E:=\sup\{\|\chi_F\|_E : F\subseteq I \text{ finite}\}\in [1,\infty].
\]
\end{definition}

\begin{definition}[Restriction of a sequence lattice]\label{def:restriction}
If $J\subseteq I$, we write $E|J$ for the restriction of $E$ to $J$:
\[
E|J := \{x|_J:\ x\in E\},\qquad \|u\|_{E|J}:=\inf\{\|x\|_E:\ x\in E,\ x|_J=u\}.
\]
\end{definition}

\begin{lemma}\label{lem:restriction-properties}
Let $E$ be a Banach sequence lattice on $I$ and let $J\subseteq I$.  Then $E|J$ is a Banach
sequence lattice on $J$.  In particular, $c_{00}(J)$ is dense in $E|J$, the inclusion
$E|J\hookrightarrow \ell_\infty(J)$ is contractive, and
\[
C_{E|J}=\sup\{\|\chi_F\|_E:\ F\subseteq J\text{ finite}\}.
\]
Here, on the right-hand side, $\chi_F$ is regarded as a function on $I$ by extending it by zero
outside $F$.
\end{lemma}

\begin{proof}
Let $R_J\colon E\to \ell_\infty(J)$ be the restriction map.  Since
$\|R_Jx\|_\infty\le \|x\|_\infty\le \|x\|_E$, the kernel of $R_J$ is closed and $E|J$, with the
quotient norm above, is canonically isometric to $E/\ker R_J$.  Hence $E|J$ is a Banach space.
The same estimate also gives the contractive inclusion $E|J\hookrightarrow \ell_\infty(J)$:
if $u\in E|J$ and $x|_J=u$, then $\|u\|_\infty\le \|x\|_E$, and taking the infimum over all
such extensions $x$ gives $\|u\|_\infty\le \|u\|_{E|J}$.

We next check solidity.  Let $v\in E|J$ and let $u\in\ell_\infty(J)$ satisfy
$|u_j|\le |v_j|$ for $j\in J$.  Given an extension $y\in E$ of $v$, define $x\in\ell_\infty(I)$
by $x_j=u_j$ for $j\in J$ and $x_i=0$ for $i\notin J$.  Then $|x_i|\le |y_i|$ for every
$i\in I$, so $x\in E$ and $\|x\|_E\le \|y\|_E$ by solidity of $E$.  Thus $u=x|_J\in E|J$ and,
after taking the infimum over all such $y$, $\|u\|_{E|J}\le \|v\|_{E|J}$.

Finally, $c_{00}(J)$ is dense in $E|J$.  Indeed, if $u\in E|J$, choose $x\in E$ with
$x|_J=u$.  Since $c_{00}(I)$ is dense in $E$, for every $\varepsilon>0$ there is
$w\in c_{00}(I)$ with $\|x-w\|_E<\varepsilon$.  Then $w|_J\in c_{00}(J)$ and
\[
\|u-w|_J\|_{E|J}\le \|x-w\|_E<\varepsilon.
\]
For finite $F\subseteq J$, the zero extension of $\chi_F$ belongs to $E$ and gives
$\|\chi_F\|_{E|J}\le \|\chi_F\|_E$.  Conversely, if $x\in E$ and $x|_J=\chi_F$, then
$|\chi_F|\le |x|$ on $I$, again viewing $\chi_F$ as zero outside $F$; hence solidity gives
$\|\chi_F\|_E\le \|x\|_E$.  Taking the infimum over all such $x$ proves
$\|\chi_F\|_{E|J}=\|\chi_F\|_E$, and the formula for $C_{E|J}$ follows.
\end{proof}

\begin{remark}\label{rem:norm-normalisation}
Assumption~\emph{(c)} is a convenient normalisation ensuring that the coordinate evaluations
$x\mapsto x_i$ have norm~$\le 1$ on $E$, and hence coordinate projections in $E$-sums are contractive.
Many concrete examples satisfy \emph{(c)} after a harmless global renorming.
\end{remark}

\begin{remark}\label{rem:CE}
The hypothesis $C_E<\infty$ is equivalent to uniform control of the $E$-norm by the supremum norm on finitely
supported families:
\[
\|x\|_E \le C_E\|x\|_\infty \qquad (x\in \czerozero(I)).
\]
Indeed, if $x$ has support contained in $F$, then $|x|\le \|x\|_\infty \chi_F$, so solidity yields
$\|x\|_E\le \|x\|_\infty \|\chi_F\|_E\le C_E\|x\|_\infty$.
\end{remark}

\begin{lemma}\label{lem:CE-equivalence}
Let $E$ be a Banach sequence lattice on $I$ with $C_E<\infty$.
Then $E=c_0(I)$ as sets, and for all $x\in E$ one has
\[
\|x\|_\infty\le \|x\|_E\le C_E\|x\|_\infty.
\]
\end{lemma}

\begin{proof}
The inequality $\|x\|_\infty\le \|x\|_E$ is Definition~\ref{def:sequence-lattice}(c).
For the upper bound, fix $x\in E$ and $\varepsilon>0$.
By Definition~\ref{def:sequence-lattice}(a), $c_{00}(I)$ is dense in~$E$, so there exists
$v\in c_{00}(I)$ with $\|x-v\|_E<\varepsilon$.
Let $F:=\supp(v)$. Since $v$ vanishes on $I\setminus F$, we have
$x\chi_{I\setminus F}=(x-v)\chi_{I\setminus F}$, and hence by solidity
\[
\|x-x\chi_F\|_E = \|x\chi_{I\setminus F}\|_E \le \|x-v\|_E < \varepsilon.
\]
Then, using Remark~\ref{rem:CE} for $x\chi_F\in c_{00}(I)$,
\[
\|x\|_E \le \|x-x\chi_F\|_E+\|x\chi_F\|_E
\le \varepsilon + C_E\|x\chi_F\|_\infty
\le \varepsilon + C_E\|x\|_\infty.
\]
Letting $\varepsilon\downarrow 0$ gives $\|x\|_E\le C_E\|x\|_\infty$.

Finally, $E\subseteq c_0(I)$ because if $x\in E$ and $\varepsilon>0$, pick $v\in c_{00}(I)$ with
$\|x-v\|_E<\varepsilon$; then for $i\notin\supp(v)$ we have
$|x_i|=|x_i-v_i|\le \|x-v\|_\infty\le \|x-v\|_E<\varepsilon$.
Conversely, if $y\in c_0(I)$ then $(y\chi_F)_F\subset c_{00}(I)\subset E$ is Cauchy in $E$ because
$\|y\chi_F-y\chi_G\|_E\le C_E\|y\chi_F-y\chi_G\|_\infty$, and $y\chi_F\to y$ in $\|\cdot\|_\infty$.
Thus $y\in E$.
\end{proof}

\begin{lemma}\label{lem:pointwise}
Let $E$ be a Banach sequence lattice on~$I$.
Then $E$ is a Banach algebra for pointwise multiplication, and
\[
\|xy\|_E\le \|x\|_E\,\|y\|_E \qquad (x,y\in E).
\]
\end{lemma}

\begin{proof}
Let $x\in E$ and $y\in \ell_\infty(I)$. Then $|xy|\le \|y\|_\infty |x|$ pointwise, so by solidity
$xy\in E$ and
\[
\|xy\|_E \le \|\|y\|_\infty x\|_E = \|y\|_\infty \|x\|_E.
\]
Thus $E$ is an ideal in $\ell_\infty(I)$, in the sense that $xy\in E$ whenever $x\in E$ and $y\in\ell_\infty(I)$. If now $x,y\in E$, then by Definition~\ref{def:sequence-lattice}(c)
we have $\|y\|_\infty\le \|y\|_E$, and hence
\[
\|xy\|_E\le \|x\|_E\|y\|_\infty \le \|x\|_E\|y\|_E,
\]
which is the desired submultiplicativity.
\end{proof}

We can now define unconditional sums of Banach spaces and, in particular, Banach algebras.

\begin{definition}[$E$-sums]\label{def:Esum}
Let $I$ be a non-empty set, let $E$ be a Banach sequence lattice on~$I$, and let $(X_i)_{i\in I}$ be a family of Banach spaces.
The \emph{$E$-sum} $X:=\big(\bigoplus_{i\in I}X_i\big)_{\!E}$ consists of all families $x=(x_i)_{i\in I}$ with $x_i\in X_i$ such that
\[
\bigl(\|x_i\|\bigr)_{i\in I}\in E,
\]
equipped with the norm
\[
\|x\|:=\bigl\|\bigl(\|x_i\|\bigr)_{i\in I}\bigr\|_E.
\]
If each $X_i$ is a Banach algebra, we endow $X$ with coordinatewise multiplication:
$(x_iy_i)_{i\in I}$.
\end{definition}

\begin{proposition}\label{prop:Esum-algebra}
Let $E$ be a Banach sequence lattice on~$I$ and let $(A_i)_{i\in I}$ be a family of Banach algebras.
Then $A:=\big(\bigoplus_{i\in I}A_i\big)_{\!E}$ is a Banach algebra with coordinatewise multiplication, and
\[
\|ab\|\le \|a\|\,\|b\| \qquad (a,b\in A).
\]
Moreover, for each $i\in I$, the coordinate projection $\pi_i\colon A\to A_i$ is a contractive homomorphism,
and the coordinate embedding $\iota_i\colon A_i\to A$ is a bounded homomorphism with
\[
\|\iota_i\|=\|\delta_i\|_E.
\]
\end{proposition}

\begin{proof}
Let $a=(a_i)$ and $b=(b_i)$ lie in~$A$.
Then $u:=(\|a_i\|)$ and $v:=(\|b_i\|)$ belong to~$E$.
For each $i$ we have $\|a_ib_i\|\le \|a_i\|\,\|b_i\|$, hence
\[
(\|a_ib_i\|)_{i\in I}\le u\,v \quad\text{pointwise}.
\]
By Lemma~\ref{lem:pointwise}, $uv\in E$ and $\|uv\|_E\le \|u\|_E\|v\|_E$.
Using solidity once more yields $(\|a_ib_i\|)\in E$ and
\[
\|ab\|=\|(\|a_ib_i\|)\|_E\le \|uv\|_E\le \|u\|_E\|v\|_E=\|a\|\,\|b\|.
\]
For $\pi_i$, Definition~\ref{def:sequence-lattice}\emph{(c)} yields
\[
\|\pi_i(a)\|=\|a_i\|
\le \bigl\|(\|a_j\|)_{j\in I}\bigr\|_{\infty}
\le \bigl\|(\|a_j\|)_{j\in I}\bigr\|_{E}
=\|a\|
\qquad (a=(a_j)\in A),
\]
so $\|\pi_i\|\le 1$.
For $\iota_i$, if $b\in A_i$ then $\iota_i(b)$ is supported at $i$ and
\[
\|\iota_i(b)\|
= \bigl\|(\|b\|\delta_i)\bigr\|_{E}
= \|b\|\,\|\delta_i\|_{E},
\]
hence $\|\iota_i\|=\|\delta_i\|_{E}$.
\end{proof}

\begin{lemma}[Finite-support subalgebras are dense]\label{lem:dense-finite-support}
Let $E$ be a Banach sequence lattice on~$I$ and let $(X_i)_{i\in I}$ be Banach spaces.
For $F\in\cF(I)$ let
\[
X_F:=\{x\in (\bigoplus_{i\in I}X_i)_E : x_i=0 \text{ for } i\notin F\}.
\]
Then $\bigcup_{F\in\cF(I)}X_F$ is dense in $(\bigoplus_{i\in I}X_i)_E$.
\end{lemma}

\begin{proof}
Let $x=(x_i)$ lie in $X:=(\bigoplus_{i\in I}X_i)_E$ and set $u:=(\|x_i\|)\in E$.
Since $\czerozero(I)$ is dense in~$E$, there is $v\in\czerozero(I)$ with $\|u-v\|_E<\varepsilon$.
Let $F:=\mathrm{supp}(v)$.
Then $u\chi_F-v=u\chi_F-v\chi_F=(u-v)\chi_F$, so solidity gives
$\|u\chi_F-v\|_E\le \|u-v\|_E<\varepsilon$.
Therefore,
\[
\|x-x\chi_F\|=\|(\|x_i\|\chi_{I\setminus F})\|_E
=\|u-u\chi_F\|_E
\le \|u-v\|_E+\|v-u\chi_F\|_E <2\varepsilon,
\]
and $x\chi_F\in X_F$.
\end{proof}

\section{Amenability of unconditional $E$-sums}\label{sec:main}

We can now state and prove our main result.

\begin{theorem}\label{thm:main}
Let $I$ be a non-empty set, let $E$ be a Banach sequence lattice on~$I$, let $(A_i)_{i\in I}$ be a family of Banach algebras, and set
\[
A:=\Big(\bigoplus_{i\in I}A_i\Big)_{\!E}.
\]
Then:
\begin{enumerate}[label={\rm(\arabic*)},leftmargin=3.2em]
\item One always has
\[
\sup_{i\in I}\AM(A_i)\le \AM(A).
\]
\item Assume, in addition, that
\[
C_E=\sup\{\|\chi_F\|_E: F\in \cF(I)\}<\infty.
\]
Then $A$ is amenable if and only if $\sup_{i\in I}\AM(A_i)<\infty$.  In this case,
\[
\sup_{i\in I}\AM(A_i)\ \le\ \AM(A)\ \le\ C_E^2\,\sup_{i\in I}\AM(A_i).
\]
\end{enumerate}
\end{theorem}

\begin{proof}
For each $i\in I$, the coordinate projection $\pi_i\colon A\to A_i$ is a contractive surjective homomorphism
(Proposition~\ref{prop:Esum-algebra}).  Lemma~\ref{lem:quotient} gives
$\AM(A_i)\le \AM(A)$, hence $\sup_{i\in I}\AM(A_i)\le \AM(A)$.
This proves (1).  In particular, if $\sup_{i\in I}\AM(A_i)=+\infty$, then $\AM(A)=+\infty$ and $A$ is not amenable.

For (2), assume that $C_E<\infty$ and that $S:=\sup_{i\in I}\AM(A_i)<\infty$.
For each $F\in\cF(I)$ let $A_F\subseteq A$ denote the closed subalgebra of elements supported in~$F$.
By Lemma~\ref{lem:dense-finite-support}, $\bigcup_{F\in\cF(I)}A_F$ is dense in~$A$.

Fix $F\in\cF(I)$.
Let $B_F:=\bigoplus_{i\in F} A_i$ equipped with the $\ell^\infty$-norm and coordinatewise multiplication.
By Lemma~\ref{lem:finite-infty},
\[
\AM(B_F)=\max_{i\in F}\AM(A_i)\le S.
\]
Consider the identity map $q_F\colon B_F\to A_F$ (same underlying vectors, different norms).
For $b=(b_i)_{i\in F}\in B_F$ we have
\[
\|q_F(b)\|_{A_F}=\|(\|b_i\|)_{i\in F}\|_E \le \|b\|_\infty\,\|\chi_F\|_E \le C_E\,\|b\|_\infty,
\]
where the key step uses solidity and $|(\|b_i\|)|\le \|b\|_\infty \chi_F$.
Hence $\|q_F\|\le C_E$, and $q_F$ is a surjective homomorphism.
Lemma~\ref{lem:quotient} yields
\[
\AM(A_F)\le \|q_F\|^2\,\AM(B_F)\le C_E^2\,S.
\]
Finally, Lemma~\ref{lem:dense-union} applied to the dense family $(A_F)_{F\in\cF(I)}$ gives
\[
\AM(A)\le \sup_{F\in\cF(I)}\AM(A_F)\le C_E^2\,S<\infty,
\]
so $A$ is amenable and the stated two-sided estimate holds.  The converse implication in (2) follows from (1).
\end{proof}

\begin{remark}\label{rem:CE1}
If $C_E=1$ then Theorem~\ref{thm:main} yields the exact formula
$\AM(A)=\sup_{i\in I}\AM(A_i)$.
This occurs, for instance, for $E=c_0(I)$ with the usual supremum norm.
\end{remark}

The next proposition shows that the finiteness of~$C_E$ is close to optimal in typical situations.

\begin{proposition}\label{prop:CE-necessary}
Let $E$ be a Banach sequence lattice on~$I$, let $(A_i)_{i\in I}$ be a family of Banach algebras, and set
\[
A:=\Big(\bigoplus_{i\in I}A_i\Big)_{\!E}.
\]
Assume that $A$ has a bounded right approximate identity $(e_\alpha)_\alpha$, and put
\[
M:=\sup_\alpha \|e_\alpha\|<\infty.
\]
Let
\[
J:=\{j\in I:\ A_j\neq\{0\}\}.
\]
Then
\[
\sup\bigl\{\|\chi_F\|_E:\ F\in\cF(I),\ F\subseteq J\bigr\}=C_{E|J}\le M<\infty.
\]
In particular, if $A$ is amenable then $C_{E|J}<\infty$; and if moreover
$A_i\neq\{0\}$ for all $i\in I$, then $C_E\le M<\infty$.
\end{proposition}

\begin{proof}
Fix $j\in J$ and choose $a_j\in A_j$ with $\|a_j\|=1$.
Let $u_j\in A$ be supported at~$j$ with coordinate $u_j(j)=a_j$.
Since $(e_\alpha)$ is a right approximate identity,
\[
\|u_j-u_je_\alpha\|\longrightarrow 0.
\]
Because multiplication is coordinatewise, $(u_je_\alpha)(j)=a_j\,\pi_j(e_\alpha)$, hence
\[
\|a_j-a_j\pi_j(e_\alpha)\|\longrightarrow 0 \qquad (j\in J).
\]

Now fix $F\in\cF(I)$ with $F\subseteq J$ and let $0<\delta<1$.
Choose $\alpha$ so large that $\|a_j-a_j\pi_j(e_\alpha)\|\le \delta$ for all $j\in F$.
Then
\[
\|a_j\pi_j(e_\alpha)\|\ge \|a_j\|-\|a_j-a_j\pi_j(e_\alpha)\|\ge 1-\delta
\qquad (j\in F).
\]
By submultiplicativity, $\|a_j\pi_j(e_\alpha)\|\le \|a_j\|\,\|\pi_j(e_\alpha)\|=\|\pi_j(e_\alpha)\|$,
so $\|\pi_j(e_\alpha)\|\ge 1-\delta$ for all $j\in F$.

Let $P_F\colon A\to A$ denote the coordinate projection onto~$F$.  By solidity,
\[
(1-\delta)\|\chi_F\|_E
\le \bigl\|(\|\pi_j(e_\alpha)\|)_{j\in F}\bigr\|_E
=\|P_F(e_\alpha)\|
\le \|e_\alpha\|\le M.
\]
Thus $\|\chi_F\|_E\le M/(1-\delta)$; letting $\delta\downarrow0$ gives
$\|\chi_F\|_E\le M$.  Taking the supremum over all such $F$ proves the claim.
\end{proof}

\begin{remark}\label{rem:main-necessity-bridge}
Theorem~\ref{thm:main} gives a sufficiency criterion once $C_E<\infty$ is assumed.
Proposition~\ref{prop:CE-necessary} explains why this hypothesis is not merely technical in
the presence of infinitely many active coordinates: for any amenable $E$-sum, the same
boundedness condition must hold on the active support $J=\{i\in I:A_i\neq\{0\}\}$.
Corollary~\ref{cor:iff} is precisely the combination of these two facts.
\end{remark}

\begin{corollary}\label{cor:iff}
Let $A=(\oplus_{i\in I}A_i)_E$ and let $J:=\{i\in I: A_i\neq\{0\}\}$.
Assume that $J$ is infinite. Then $A$ is amenable if and only if
$C_{E|J}<\infty$ and $\sup_{i\in J}\AM(A_i)<\infty$.
\end{corollary}

\begin{proof}
If $A$ is amenable, then each $A_i$ is a quotient of $A$ via the coordinate projection,
so $\sup_{i\in J}\AM(A_i)\le \AM(A)<\infty$. Moreover, Proposition~\ref{prop:CE-necessary}
yields $C_{E|J}<\infty$.

Theorem~\ref{thm:main} applied to the restriction $E|J$ shows that $(\oplus_{i\in J}A_i)_{E|J}$ is amenable. Since $A$ is canonically isometrically identifiable with $(\oplus_{i\in J}A_i)_{E|J}$ (the other coordinates being $0$; the equality of the norms follows from the same solidity argument as in Lemma~\ref{lem:restriction-properties}), it follows that $A$ is amenable.
\end{proof}

\section{Consequences and examples}\label{sec:applications}

We record several consequences.

\subsection{The $c_0$-sum over arbitrary index sets}

\begin{corollary}[$c_0$-sums, possibly uncountable]\label{cor:c0}
Let $I$ be a non-empty set and let $(A_i)_{i\in I}$ be a family of Banach algebras.
Set $A:=c_0\text{-}\bigoplus_{i\in I}A_i$, i.e.\ the $E$-sum with $E=c_0(I)$ and the supremum norm.
Then
\[
\AM(A)=\sup_{i\in I}\AM(A_i).
\]
In particular, $A$ is amenable if and only if the family $(\AM(A_i))_{i\in I}$ is bounded.
\end{corollary}

\begin{proof}
For comparison, this identity is proved in \cite[Proposition~2.3.19]{Runde2020ABA}.
For $E=c_0(I)$ we have $C_E=1$, so the conclusion follows from Theorem~\ref{thm:main} and Remark~\ref{rem:CE1}.
\end{proof}

\subsection{Weighted $c_0$-sums}

Weighted $c_0$-norms provide a simple family of coefficient lattices where the
constant $C_E$ is completely explicit and captures the uniformity of the weights.

\begin{proposition}\label{prop:weighted-c0}
Let $I$ be a non-empty set and let $w=(w_i)_{i\in I}$ satisfy $w_i>0$ for all $i$.
Define
\[
E=c_0(I,w):=\Bigl\{x=(x_i)_{i\in I}\in\mathbb{C}^I:
\{i\in I:\ w_i|x_i|\ge \varepsilon\}\in \cF(I)\ \text{for every }\varepsilon>0\Bigr\},
\]
with norm
\[
\|x\|_E:=\sup_{i\in I} w_i|x_i|.
\]
Then:
\begin{enumerate}[label={\rm(\alph*)}]
\item One has $E\subseteq \ell_\infty(I)$ if and only if $\inf_{i\in I} w_i>0$.
Moreover, $E$ is a Banach sequence lattice on $I$ in the sense of
Definition~\ref{def:sequence-lattice} if and only if $\inf_{i\in I} w_i\ge 1$.
\item For every finite $F\subset I$,
\[
\|\chi_F\|_E=\max_{i\in F} w_i,
\qquad\text{and hence}\qquad
C_E=\sup_{i\in I} w_i\in[1,\infty].
\]
\item Assume in addition that $\inf_{i\in I} w_i\ge 1$ and let
\(
A=\bigl(\bigoplus_{i\in I} A_i\bigr)_E
\)
be the corresponding $E$-sum Banach algebra. Then
\[
\sup_{i\in I}\AM(A_i)\ \le\ \AM(A)\ \le\ C_E^2\,\sup_{i\in I}\AM(A_i)
\ =\ \Bigl(\sup_{i\in I} w_i\Bigr)^2\,\sup_{i\in I}\AM(A_i).
\]
In particular, if $\sup_i w_i<\infty$ and $\sup_i\AM(A_i)<\infty$, then $A$ is amenable.
If $I$ is infinite and each $A_i$ is non-zero, then amenability of $A$ forces $\sup_i w_i<\infty$.
\end{enumerate}
\end{proposition}

\begin{proof}
(a) If $m:=\inf_i w_i>0$ and $x\in E$, then $|x_i|\le m^{-1}\|x\|_E$ for all $i$, so $x\in\ell_\infty(I)$.
Conversely, if $\inf_i w_i=0$, pick $(i_n)$ with $w_{i_n}<1/n^2$ and set $x_{i_n}=n$ and $x_i=0$ otherwise.
Then $w_i|x_i|\to 0$ but $x\notin\ell_\infty(I)$, so $E\nsubseteq \ell_\infty(I)$.
The inclusion $E\hookrightarrow\ell_\infty(I)$ is contractive, i.e.\ $\|x\|_\infty\le \|x\|_E$ for all $x\in E$,
if and only if $w_i\ge 1$ for all $i$ (test $x=\delta_i$ and use $|x_i|\le w_i|x_i|$).
This yields the characterisation of when $E$ is a Banach sequence lattice.

(b) For finite $F$, $\chi_F(i)=1$ on $F$ and $0$ otherwise, hence
\(
\|\chi_F\|_E=\sup_i w_i|\chi_F(i)|=\max_{i\in F} w_i
\),
and taking the supremum over $F$ gives $C_E=\sup_i w_i$.

(c) Under $\inf_i w_i\ge 1$ we have $E$ a Banach sequence lattice and $C_E=\sup_i w_i$ by (a)--(b).
The estimate for $\AM(A)$ is then exactly Theorem~\ref{thm:main}.
If moreover $I$ is infinite and all $A_i$ are non-zero, amenability of $A$ implies that $A$ has a bounded approximate identity,
so Proposition~\ref{prop:CE-necessary} forces $C_E<\infty$, i.e.\ $\sup_i w_i<\infty$.
\end{proof}

\begin{remark}\label{rem:weight-convention}
With the convention $\|x\|_E=\sup_i w_i|x_i|$, one has $\|\chi_F\|_E=\max_{i\in F} w_i$ and
$C_E=\sup_i w_i$.  If instead one defines the weighted norm by
$\|x\|=\sup_i |x_i|/w_i$, then $\|\chi_F\|=\max_{i\in F} 1/w_i$ and $C_E=\sup_i 1/w_i$.
\end{remark}

\subsection{Constant families of summands}

We now single out an explicit statement for Banach-algebra-valued sequence algebras and, more generally,
for constant families of summands.  This is the unconditional $E$-sum analogue one typically wants to
\emph{see stated plainly}, even if it appears in a longer discussion (for instance within a $J$-sum setting).

\begin{corollary}\label{cor:Piszczek}
Let $I$ be a non-empty set, let $E$ be a Banach sequence lattice on~$I$ with $C_E<\infty$, and let $B$ be a Banach algebra.
Consider the constant family $A_i=B$ ($i\in I$) and set
\[
B^{(E)}:=\Big(\bigoplus_{i\in I} B\Big)_{\!E}.
\]
Then $B^{(E)}$ is amenable if and only if $B$ is amenable. Moreover,
\[
\AM(B)\ \le\ \AM\bigl(B^{(E)}\bigr)\ \le\ C_E^2\,\AM(B).
\]
In the particular case $E=c_0(I)$ one has the exact identity $\AM(B^{(c_0)})=\AM(B)$.
\end{corollary}

\begin{proof}
This is the special case of Theorem~\ref{thm:main} with $\AM(A_i)=\AM(B)$ for all~$i$.
For $E=c_0(I)$ we have $C_E=1$, hence equality.
\end{proof}

\begin{remark}[A tensor-product viewpoint]\label{rem:tensor-viewpoint}
Equip $E$ with pointwise multiplication (Lemma~\ref{lem:pointwise}).
For any Banach algebra $B$, the bilinear map
\[
E\times B\longrightarrow B^{(E)},\qquad (x,b)\longmapsto (x_i b)_{i\in I},
\]
is contractive and multiplicative, hence induces a contractive algebra homomorphism
from the projective tensor product $E\hatot B$ into $B^{(E)}$ whose range contains the
finitely supported sequences and therefore is dense.
In the special cases $E=\ell^1(I)$ and $E=c_0(I)$ one recovers the classical identifications
$\ell^1(I,B)\cong \ell^1(I)\hatot B$ and $c_0(I,B)\cong c_0(I)\hatot_\varepsilon B$
(with the injective tensor product $\hatot_\varepsilon$), situating
Corollary~\ref{cor:Piszczek} alongside the standard tensor-product permanence results.
\end{remark}

\subsection{Non-$c_0$ behaviour and Orlicz spaces}

A standard obstruction to amenability is the absence of a bounded approximate identity.
The next corollary illustrates how Proposition~\ref{prop:CE-necessary} rules out amenability for many
classical unconditional sequence norms when infinitely many non-zero summands are present.

\begin{corollary}\label{cor:lp-not-amenable}
Let $I$ be infinite and let $(A_i)_{i\in I}$ be non-zero Banach algebras.
If $E$ is a Banach sequence lattice on~$I$ with $C_E=+\infty$, then the $E$-sum
$A=(\bigoplus_{i\in I}A_i)_E$ is not amenable.
In particular, for $1\le p<\infty$, the algebra $(\bigoplus_{i\in I}A_i)_{\ell_p(I)}$ is not amenable.
\end{corollary}

\begin{proof}
If $A$ were amenable, it would admit a bounded approximate identity; this would force $C_E<\infty$ by
Proposition~\ref{prop:CE-necessary}, a contradiction.
For $E=\ell_p(I)$ ($1\le p<\infty$) and $I$ infinite, one has $\|\chi_F\|_{\ell_p}=|F|^{1/p}$, hence $C_E=+\infty$.
\end{proof}

\begin{remark}\label{rem:lp-normalisation}
For $1\le p<\infty$, the usual $\ell_p(I)$ norm satisfies the normalisation in
Definition~\ref{def:sequence-lattice}(c), since $\|x\|_\infty\le \|x\|_p$ for every
$x\in\ell_p(I)$.  Thus no renorming is needed in Corollary~\ref{cor:lp-not-amenable}; the
normalisation remark above is only relevant for coefficient spaces whose customary norm does
not dominate the supremum norm.
\end{remark}

\begin{proposition}[Orlicz hearts and the constant $C_E$]\label{prop:orlicz-CE}
Let $\varphi\colon [0,\infty)\to[0,\infty)$ be an Orlicz function and let $E=h_\varphi(I)$ be the Orlicz
heart (over counting measure on $I$), equipped with the Luxemburg norm
\[
\|x\|_{\varphi}:=\inf\Bigl\{\lambda>0:\ \sum_{i\in I}\varphi\bigl(|x_i|/\lambda\bigr)\le 1\Bigr\},
\]
where $\sum_{i\in I}$ denotes the supremum of finite subsums. Let $\varphi^{-1}$ denote the generalised inverse
\[
\varphi^{-1}(s):=\inf\{t\ge 0:\ \varphi(t)\ge s\}\qquad (s\ge 0).
\]
Then:
\begin{enumerate}[label={\rm(\alph*)}]
\item For each $i\in I$, $\|\delta_i\|_{\varphi}=1/\varphi^{-1}(1)$.
\item More generally, for every finite set $F\subset I$ with $|F|=n$,
\[
\|\chi_F\|_{\varphi}=\frac{1}{\varphi^{-1}(1/n)}.
\]
Consequently,
\[
C_E
=\sup_{\substack{F\subset I\\ |F|<\infty}}\|\chi_F\|_{\varphi}
=\sup_{n\in\mathbb{N}} \frac{1}{\varphi^{-1}(1/n)}\ \in\ (0,\infty]
\quad\text{provided $I$ is infinite.}
\]
\item Assume that $I$ is infinite. Then $C_E<\infty$ if and only if $\varphi$ vanishes on a neighbourhood of $0$, i.e.
\[
a_\varphi:=\sup\{t\ge 0:\ \varphi(t)=0\}>0.
\]
In that case $C_E=1/a_\varphi$. If $\varphi(t)>0$ for all $t>0$ (in particular, for any $N$-function), then $C_E=+\infty$.
\end{enumerate}
\end{proposition}

\begin{proof}
(a) For $\delta_i$, the defining inequality becomes $\varphi(1/\lambda)\le 1$, i.e.\ $1/\lambda\le \varphi^{-1}(1)$.

(b) If $F$ is finite with $|F|=n$, then
\[
\sum_{i\in I}\varphi\bigl(|\chi_F(i)|/\lambda\bigr)=\sum_{i\in F}\varphi(1/\lambda)=n\,\varphi(1/\lambda).
\]
Thus $\|\chi_F\|_\varphi$ is the infimum of $\lambda>0$ such that $n\,\varphi(1/\lambda)\le 1$,
equivalently $\varphi(1/\lambda)\le 1/n$, i.e.\ $1/\lambda\le \varphi^{-1}(1/n)$.

(c) Since $1/n\downarrow 0$ and $\varphi^{-1}$ is non-decreasing,
\[
\inf_{n\in\mathbb N}\varphi^{-1}(1/n)=\lim_{n\to\infty}\varphi^{-1}(1/n)=\lim_{s\downarrow 0}\varphi^{-1}(s)=:a_\varphi.
\]
Hence $C_E<\infty$ iff $a_\varphi>0$, and then $\sup_n 1/\varphi^{-1}(1/n)=1/a_\varphi$.
If $\varphi(t)>0$ for all $t>0$, then $\varphi^{-1}(s)\downarrow 0$ as $s\downarrow 0$, so $C_E=\infty$.
\end{proof}

\begin{remark}[Connection with $\Delta_2$ at $0$]\label{rem:orlicz-delta2}
For Orlicz sequence spaces, the condition that $c_{00}(I)$ be dense in $\ell_\varphi(I)$ is
equivalent to a $\Delta_2$-type condition at $0$ on $\varphi$ (and in any case $h_\varphi(I)$ is,
by definition, the closure of $c_{00}(I)$).  The constant $C_E$ is typically infinite for the usual
Orlicz functions (\emph{e.g.}, \ $N$-functions), since these satisfy $\varphi(t)>0$ for all $t>0$.
\end{remark}

\begin{remark}[Normalisation]\label{rem:orlicz-normalisation}
To fit Definition~\ref{def:sequence-lattice} exactly, one may multiply the Luxemburg norm by a fixed scalar so that the inclusion into $\ell_\infty(I)$ is contractive.  This harmless renorming preserves the underlying algebra and changes $C_E$ only by the same scalar; in particular, it does not affect whether $C_E$ is finite.
\end{remark}

\begin{remark}\label{rem:weak-amenability}
Corollary~\ref{cor:lp-not-amenable} and Proposition~\ref{prop:orlicz-CE} highlight the rigidity of \emph{Johnson}
amenability for sequence-type constructions: for typical spaces (like $\ell_p$), the growth of $\|\chi_F\|$
forces the failure of a bounded approximate identity
(and hence amenability) as soon as infinitely many non-zero coordinates are present.
By contrast, \emph{weak amenability} (vanishing of continuous derivations into the dual)
may persist for Banach-algebra-valued sequence algebras under additional hypotheses;
see, for instance, \cite{KoczorowskiPiszczek2024WeakAmenability}, where weak amenability of
$c_0(A)$ and $\ell_p(A)$ is characterised for arbitrary Banach algebras $A$.
\end{remark}

\subsection{Sharpness of the factor $C_E^2$}

\begin{proposition}\label{prop:CE2-sharp}
For $n\in\mathbb N$, let $E=\ell_2^n$ (with its usual norm) viewed as a Banach
sequence lattice on $\{1,\dots,n\}$, and let $A_i=\mathbb C$ for $i=1,\dots,n$.
Then
\[
A:=\Big(\bigoplus_{i=1}^n A_i\Big)_E \cong (\mathbb C^n,\ \|\cdot\|_2)
\quad\text{and}\quad
\AM(A)=n=C_E^2.
\]
In particular, the constant $C_E^2$ in Theorem~\ref{thm:main}
(and hence in Corollary~\ref{cor:Piszczek}) cannot, in general, be replaced by
$C_E$.
\end{proposition}

\begin{proof}
Identify $A$ with $\mathbb C^n$ with pointwise multiplication and $\ell_2$-norm.
Write $e_1,\dots,e_n$ for the standard coordinate idempotents.

\smallskip
\noindent\emph{Step 1: uniqueness of the diagonal.}
Let $d=\sum_{i,j=1}^n \alpha_{ij}\, e_i\otimes e_j \in A\hat\otimes A$
satisfy $a\cdot d=d\cdot a$ for all $a\in A$.
For $a=(a_1,\dots,a_n)$ we have
\[
a\cdot d - d\cdot a
= \sum_{i,j=1}^n \alpha_{ij}\,(a_i-a_j)\,e_i\otimes e_j,
\]
so $\alpha_{ij}=0$ whenever $i\neq j$. Hence $d=\sum_{i=1}^n \alpha_{ii}\,e_i\otimes e_i$.
The diagonal identity $\pi(d)=1=\sum_{i=1}^n e_i$ forces $\alpha_{ii}=1$,
so the (necessarily unique) diagonal is
\[
d=\sum_{i=1}^n e_i\otimes e_i.
\]

\smallskip
\noindent\emph{Step 2: computation of $\|d\|_\gamma$.}
Since $\|e_i\|_2=1$ for each $i$, we have the crude bound
\(
\|d\|_\gamma \le \sum_{i=1}^n \|e_i\|_2^2 = n.
\)
For the reverse inequality, use the canonical isometric identification
$(A\hat\otimes A)^*\cong \mathcal B(A,A^*)$.
Let $J:A\to A^*$ be the Riesz map $J(x)=\langle \,\cdot\,,x\rangle$ associated
with the $\ell_2$ inner product; then $\|J\|=1$.
Denote by $\phi_J\in (A\hat\otimes A)^*$ the functional corresponding to $J$.
We obtain
\[
\langle d,\phi_J\rangle
= \sum_{i=1}^n \langle e_i,e_i\rangle
= n,
\]
so $\|d\|_\gamma \ge |\langle d,\phi_J\rangle| = n$.
Hence $\|d\|_\gamma=n$, and since $d$ is the unique diagonal, $\AM(A)=\|d\|_\gamma=n$.

Finally, $C_E=\sup_{F\subset\{1,\dots,n\}}\|\chi_F\|_2=\sqrt{n}$, so $n=C_E^2$.
\end{proof}

\section{James-type $J$-sums of Banach algebras}\label{sec:Jsums}

The classical James space $J$ \cite{James1960} is the first explicit example of a
non-reflexive Banach space isomorphic to its bidual.
Lindenstrauss \cite{Lindenstrauss1971} showed that James' construction can be
adapted to realise a given separable Banach space as a quotient $Y^{**}/Y$ for a
suitable $Y$ with a shrinking monotone basis.  Bellenot \cite{Bellenot1982}
introduced a flexible framework---the \emph{$J$-sum}---subsuming both James' and
Lindenstrauss' constructions.  A Banach-algebraic variant was studied by Andrew
and Green \cite{AndrewGreen1980} (for $J$ itself) and, for $J$-sums, by Kiran and
Singh \cite{KiranSingh1988}.

In this section we recall the Bellenot--Lindenstrauss definition, record the
basic structural facts needed later, and explain how our results on unconditional
$E$-sums apply \emph{even when such sums appear only as building blocks inside a
conditional $J$-sum}.  

\subsection{The Bellenot construction}\label{subsec:Jsum-def}

Let $\N_0:=\N\cup\{0\}$, let $\mathcal{P}$ denote the family of all non-empty
finite subsets of $\N_0$, and let $(X_n,\|\cdot\|_n)_{n\in\N_0}$ be a sequence of
Banach spaces with $\dim X_0=0$ (so $X_0=\{0\}$).
Let $\Phi=(\varphi_n)_{n\in\N_0}$ be a sequence of operators
$\varphi_n\colon X_n\to X_{n+1}$ with $\|\varphi_n\|\le 1$ for all $n\in\N_0$.
For $n\le m$ in $\N_0$ set
\[
\varphi_m^n :=
\begin{cases}
\mathrm{Id}_{X_n}, & m=n,\\
\varphi_{m-1}\circ\cdots\circ \varphi_{n+1}\circ \varphi_n, & m>n.
\end{cases}
\]
Write $\Pi:=\prod_{n\in\N_0} X_n$ for the Cartesian product.

\begin{definition}[{\cite[Def.\ 2.1]{Bellenot1982}}]\label{def:J-norm}
For $x=(x_n)_{n\in\N_0}\in\Pi$ and $S=\{p_0<\cdots<p_k\}\in\mathcal{P}$ define
\begin{align}
\sigma(x,S)
 &:= \Big( \sum_{i=1}^k \big\|\varphi_{p_i}^{p_{i-1}}(x_{p_{i-1}})-x_{p_i}\big\|_{p_i}^2 \Big)^{1/2},\label{eq:sigma}\\
\rho(x,S)
 &:= \Big( \sigma(x,S)^2 + \|x_{p_k}\|_{p_k}^2 \Big)^{1/2}.\label{eq:rho}
\end{align}
Then
\begin{equation}\label{eq:Jnorm}
\|x\|_{J}:=\frac{1}{\sqrt{2}}\sup_{S\in\mathcal{P}} \rho(x,S)\in[0,\infty].
\end{equation}
Let $c_{00}(\Phi)\subseteq\Pi$ be the space of finitely non-zero sequences.
The \emph{$J$-sum} $J(\Phi)$ is the completion of $(c_{00}(\Phi),\|\cdot\|_{J})$.
We also set
\[
\widehat{J}(\Phi):=\{x\in\Pi:\ \|x\|_{J}<\infty\}.
\]
\end{definition}

\begin{lemma}\label{lem:J-elementary}
For $x\in \widehat{J}(\Phi)$ and $S\in\mathcal{P}$ one has
\[
\sigma(x,S)\le \rho(x,S)\le \sqrt{2}\,\|x\|_{J}.
\]
Moreover, for every $n\in\N$ we have the coordinate estimate
\begin{equation}\label{eq:coord-contract}
\|x_n\|_n \le \|x\|_{J}.
\end{equation}
\end{lemma}

\begin{proof}
The first inequality is immediate from \eqref{eq:rho}.
The second follows from \eqref{eq:Jnorm}.
For \eqref{eq:coord-contract}, fix $n\in\N$ and take $S=\{0,n\}$.
Since $x_0=0$ and $\varphi_n^0(x_0)=0$, \eqref{eq:sigma} gives
$\sigma(x,S)=\|x_n\|_n$, hence $\rho(x,S)=\sqrt{2}\|x_n\|_n$.
Now \eqref{eq:Jnorm} yields $\|x\|_{J}\ge \frac1{\sqrt2}\rho(x,S)=\|x_n\|_n$.
\end{proof}

Bellenot showed that $(X_n)_{n\in\N}$ forms a bimonotone decomposition of
$J(\Phi)$; see, \emph{e.g.}, \cite[Prop.\ 2.1]{GonzalezPello2023} for a clean modern
account.

\begin{proposition}[Bimonotone decomposition {\cite{Bellenot1982}}]\label{prop:bimonotone}
For each $m\in\N$ the map $X_m\to J(\Phi)$ sending $u\in X_m$ to the sequence
$(x_n)$ with $x_m=u$ and $x_n=0$ for $n\neq m$ is an isometric embedding.
With this identification, $(X_n)_{n\in\N}$ is a bimonotone decomposition of
$J(\Phi)$; in particular,
\[
\Big\|\sum_{i=p+1}^{p+q} x_i\Big\|_J \le \Big\|\sum_{i=1}^{p+q+r} x_i\Big\|_J
\qquad
(x\in J(\Phi),\ p,q,r\in\N_0),
\]
and hence $\|x\|_J=\sup_{n\in\N}\big\|\sum_{i=1}^n x_i\big\|_J$ for $x\in J(\Phi)$.
\end{proposition}

The ambient space $\widehat{J}(\Phi)$ admits a canonical quotient description
in terms of ``eventually coherent'' sequences.

\begin{definition}\label{def:Omega}
Let $\Omega(\Phi)$ be the space of \emph{eventually coherent} sequences
\[
\Omega(\Phi)
:=
\Big\{x=(x_n)\in\Pi:\ \exists n_0\in\N\ \forall m\ge n_0:\ \varphi_m(x_m)=x_{m+1}\Big\}.
\]
For $x\in\Omega(\Phi)$ the limit $\lim_{n\to\infty}\|x_n\|_n$ exists, and
\[
\|x\|_{\Omega}:=\lim_{n\to\infty}\|x_n\|_n
\]
defines a seminorm on $\Omega(\Phi)$.
Let $\widetilde{\Omega}(\Phi)$ denote the completion of
$\Omega(\Phi)/\ker\|\cdot\|_{\Omega}$.
\end{definition}
\medskip
\noindent
We will use below that, in the Banach-algebra setting of Subsection~\ref{subsec:Jsum-algebras},
$\Omega(\Phi)$ is a subalgebra and $\|\cdot\|_{\Omega}$ is submultiplicative; see Proposition~\ref{prop:Omega-algebra}.

\begin{theorem}[Bellenot {\cite[Thm.\ 1.1]{Bellenot1982}}]\label{thm:Bellenot-quot}
\begin{enumerate}
\item $\Omega(\Phi)$ is dense in $\widehat{J}(\Phi)$.
\item The natural map $\Omega(\Phi)\to \widetilde{\Omega}(\Phi)$ extends uniquely
      to a quotient map
      \[
      \Theta\colon \widehat{J}(\Phi)\twoheadrightarrow \widetilde{\Omega}(\Phi)
      \]
      whose kernel is $J(\Phi)$.
\item If each $X_n$ is reflexive, then $J(\Phi)^{**}=\widehat{J}(\Phi)$ isometrically,
      hence $J(\Phi)^{**}/J(\Phi)$ is isometric to $\widetilde{\Omega}(\Phi)$.
\end{enumerate}
\end{theorem}

\subsection{$J$-sums as Banach algebras}\label{subsec:Jsum-algebras}

We now assume that each $X_n$ is a Banach algebra $(A_n,\|\cdot\|_n)$ and that
each bonding map $\varphi_n\colon A_n\to A_{n+1}$ is a contractive algebra
homomorphism.

\begin{definition}\label{def:Jsum-algebra}
Under the preceding hypotheses, we endow $c_{00}(\Phi)$ with coordinatewise
multiplication
\[
(x_n)\cdot(y_n) := (x_n y_n)\qquad (x,y\in c_{00}(\Phi)).
\]
We call the resulting completion $J(\Phi)$ the \emph{$J$-sum Banach algebra}
associated with $(A_n,\varphi_n)$.
Similarly, $\widehat{J}(\Phi)$ and $\Omega(\Phi)$ become algebras by
coordinatewise multiplication.
\end{definition}

\begin{lemma}[Continuity of multiplication]\label{lem:J-mult}
For $x,y\in c_{00}(\Phi)$ and $S=\{p_0<\cdots<p_k\}\in\mathcal{P}$ one has
\begin{equation}\label{eq:rho-product}
\rho(xy,S)
\le
\|y\|_{\infty}\,\rho(x,S) + \|x\|_{\infty}\,\rho(y,S),
\end{equation}
where $\|x\|_{\infty}:=\sup_{n\in\N}\|x_n\|_n$.
Consequently, coordinatewise multiplication is continuous on $\widehat{J}(\Phi)$ and on $J(\Phi)$, and
\begin{equation}\label{eq:J-submult}
\|xy\|_{J} \le 2\,\|x\|_{J}\,\|y\|_{J}
\qquad (x,y\in J(\Phi)).
\end{equation}
\end{lemma}

\begin{proof}
Fix $S=\{p_0<\cdots<p_k\}$.  Since each $\varphi_{p_i}^{p_{i-1}}$ is a homomorphism,
\[
\varphi_{p_i}^{p_{i-1}}(x_{p_{i-1}}y_{p_{i-1}})-x_{p_i}y_{p_i}
=
\big(\varphi_{p_i}^{p_{i-1}}(x_{p_{i-1}})-x_{p_i}\big)\,\varphi_{p_i}^{p_{i-1}}(y_{p_{i-1}})
+
x_{p_i}\big(\varphi_{p_i}^{p_{i-1}}(y_{p_{i-1}})-y_{p_i}\big).
\]
Taking norms, using $\|\varphi_{p_i}^{p_{i-1}}\|\le 1$ and the submultiplicativity
of $\|\cdot\|_{p_i}$ gives
\[
\big\|\varphi_{p_i}^{p_{i-1}}(x_{p_{i-1}}y_{p_{i-1}})-x_{p_i}y_{p_i}\big\|_{p_i}
\le
\|y\|_{\infty}\,\big\|\varphi_{p_i}^{p_{i-1}}(x_{p_{i-1}})-x_{p_i}\big\|_{p_i}
+
\|x\|_{\infty}\,\big\|\varphi_{p_i}^{p_{i-1}}(y_{p_{i-1}})-y_{p_i}\big\|_{p_i}.
\]
For the terminal term in \eqref{eq:rho}, we also have
\[
\|x_{p_k}y_{p_k}\|_{p_k}
\le \|y\|_\infty\|x_{p_k}\|_{p_k}+\|x\|_\infty\|y_{p_k}\|_{p_k}.
\]
Applying Minkowski's inequality to the $\ell_2$-sum consisting of the $k$ difference terms and this terminal term gives
\eqref{eq:rho-product}.

By Lemma~\ref{lem:J-elementary}, $\|x\|_\infty\le\|x\|_J$, $\|y\|_\infty\le\|y\|_J$, and
$\rho(x,S)\le\sqrt2\|x\|_J$, $\rho(y,S)\le\sqrt2\|y\|_J$ for each $S\in\mathcal P$.
Thus \eqref{eq:rho-product} gives
\[
\rho(xy,S)\le 2\sqrt2\,\|x\|_J\|y\|_J.
\]
Taking the supremum over $S$ and using \eqref{eq:Jnorm} yields
\eqref{eq:J-submult} on $c_{00}(\Phi)$, and hence multiplication extends continuously to $J(\Phi)$.
The same estimate, applied directly to arbitrary $x,y\in\widehat{J}(\Phi)$, shows that
$xy\in\widehat{J}(\Phi)$ and that multiplication is continuous there as well.
\end{proof}

\begin{remark}\label{rem:J-mult-constant}
We do not claim that the constant $2$ in Lemma~\ref{lem:J-mult} is optimal.  It is the
constant obtained directly from the Bellenot normalisation \eqref{eq:Jnorm}, which is kept here
because it makes the coordinate embeddings isometric and is the normalisation used in the
quoted structural results.  Of course, if one replaces $\|\cdot\|_J$ by the equivalent norm
$2\|\cdot\|_J$, the displayed estimate becomes contractive.  Such a rescaling changes only
norm constants, not the algebraic amenability or weak amenability questions considered in this
section.
\end{remark}

\begin{proposition}\label{prop:Omega-algebra}
Assume the hypotheses of Definition~\ref{def:Jsum-algebra}. Then:
\begin{enumerate}[label={\rm(\alph*)}]
\item $\Omega(\Phi)$ is a subalgebra of $\widehat{J}(\Phi)$ for coordinatewise multiplication, and
      $\ker\|\cdot\|_{\Omega}$ is a (two-sided) ideal in $\Omega(\Phi)$.
\item The seminorm $\|\cdot\|_{\Omega}$ is submultiplicative on $\Omega(\Phi)$.
      Consequently, $\widetilde{\Omega}(\Phi)$ is a Banach algebra.
\item The quotient map $\Theta\colon \widehat{J}(\Phi)\twoheadrightarrow \widetilde{\Omega}(\Phi)$
      from Theorem~\ref{thm:Bellenot-quot} is an algebra homomorphism.
\end{enumerate}
In particular, one has a short exact sequence of Banach algebras
\[
0\longrightarrow J(\Phi)\longrightarrow \widehat{J}(\Phi)
\overset{\Theta}{\longrightarrow}\widetilde{\Omega}(\Phi)\longrightarrow 0.
\]
\end{proposition}

\begin{proof}
(a) Let $x,y\in\Omega(\Phi)$. Choose $n_0$ such that $\varphi_m(x_m)=x_{m+1}$ and $\varphi_m(y_m)=y_{m+1}$
for all $m\ge n_0$. Since each $\varphi_m$ is an algebra homomorphism,
\[
\varphi_m(x_my_m)=\varphi_m(x_m)\varphi_m(y_m)=x_{m+1}y_{m+1}\qquad (m\ge n_0),
\]
so $xy\in\Omega(\Phi)$. If $x\in\ker\|\cdot\|_{\Omega}$ and $y\in\Omega(\Phi)$, then
$\|x_n\|_n\to 0$ and $\|y_n\|_n\to \|y\|_{\Omega}<\infty$, hence
$\|x_ny_n\|_n\le \|x_n\|_n\|y_n\|_n\to 0$, and similarly $\|y_nx_n\|_n\to 0$, so $\ker\|\cdot\|_{\Omega}$
is a two-sided ideal.

(b) For $x,y\in\Omega(\Phi)$, the limits $\|x\|_\Omega=\lim_n\|x_n\|_n$ and $\|y\|_\Omega=\lim_n\|y_n\|_n$ exist.
By coordinatewise submultiplicativity, $\|x_ny_n\|_n\le \|x_n\|_n\|y_n\|_n$ for all $n$, and taking limits gives
\[
\|xy\|_\Omega=\lim_n\|x_ny_n\|_n \le \Big(\lim_n\|x_n\|_n\Big)\Big(\lim_n\|y_n\|_n\Big)=\|x\|_\Omega\|y\|_\Omega.
\]
Thus $\|\cdot\|_\Omega$ is submultiplicative, so the induced norm on $\Omega(\Phi)/\ker\|\cdot\|_\Omega$
and its completion $\widetilde{\Omega}(\Phi)$ are submultiplicative.

(c) The natural map $\Omega(\Phi)\to\widetilde{\Omega}(\Phi)$ is multiplicative by (a)--(b), hence extends by continuity
to an algebra homomorphism $\Theta$ on $\widehat{J}(\Phi)$ via Theorem~\ref{thm:Bellenot-quot}.
\end{proof}

A key feature for applications is that \emph{each coordinate algebra is a
quotient of the $J$-sum}.

\begin{proposition}\label{prop:coord-quot}
For every $n\in\N$ the coordinate projection
\[
\pi_n\colon J(\Phi)\to A_n,\qquad \pi_n((x_k)_{k\in\N_0})=x_n,
\]
is a contractive surjective homomorphism.
\end{proposition}

\begin{proof}
Surjectivity follows because $A_n$ is isometrically embedded in $J(\Phi)$ as the
subspace of sequences supported at $n$ (Proposition~\ref{prop:bimonotone}).
Contractivity is \eqref{eq:coord-contract}.  The homomorphism property is
coordinatewise.
\end{proof}

\begin{corollary}[Coordinate obstruction to amenability]\label{cor:amenable-coordinates}
If $J(\Phi)$ is amenable, then each $A_n$ is amenable and
\[
\sup_{n\in\N}\mathrm{AM}(A_n)\le \mathrm{AM}(J(\Phi)),
\]
where $\mathrm{AM}(\cdot)$ denotes the amenability constant.  Equivalently, if some
coordinate algebra $A_n$ is not amenable, or if the numbers $\mathrm{AM}(A_n)$ are
unbounded, then $J(\Phi)$ is not amenable.
\end{corollary}

\begin{proof}
Amenability passes to quotients, and the estimate for amenability constants under
quotients gives $\mathrm{AM}(A_n)\le \|\pi_n\|^2\mathrm{AM}(J(\Phi))$.
Now use Proposition~\ref{prop:coord-quot} and take the supremum over $n$.
\end{proof}

\subsection{Unconditional quotient obstructions for $J$-sums}\label{subsec:Esums-inside-J}

The preceding corollary is only a necessary condition.  We are not aware of an
infinite non-zero Bellenot $J$-sum Banach algebra that is amenable, and the results
below should be read as obstructions rather than as a supply of examples.  The
point is that a conditional $J$-sum has contractive coordinate quotients, so any
unconditional structure present in such a quotient must satisfy the criteria from
Sections~\ref{sec:main} and~\ref{sec:applications}.

\begin{corollary}\label{cor:Esum-detected}
Assume that, for some $n\in\N$, the Banach algebra $A_n$ is isomorphic as a Banach algebra to an
$E$-sum
\[
A_n \cong \Big(\bigoplus_{i\in I} B_i\Big)_{E}
\]
for a (possibly uncountable) index set $I$, where $E$ is a Banach sequence lattice
on~$I$ in the sense of Definition~\ref{def:sequence-lattice}.  If $J(\Phi)$ is amenable, then the above
$E$-sum is amenable.  Consequently:
\begin{enumerate}[label={\rm(\alph*)}]
\item if there exists an infinite subset $J\subseteq I$ such that each $B_i$ ($i\in J$) is
      non-zero and
      \[
      \sup\bigl\{\|\chi_F\|_E:\ F\in\cF(I),\ F\subseteq J\bigr\}=+\infty,
      \]
      for example $E=\ell_p(I)$ with $1\le p<\infty$ and $J$ infinite, then $J(\Phi)$ is not amenable;
\item amenability of $J(\Phi)$ forces
      $\sup_{i\in I}\mathrm{AM}(B_i)<\infty$.
\end{enumerate}
\end{corollary}

\begin{proof}
By Proposition~\ref{prop:coord-quot}, $A_n$ is a contractive quotient of
$J(\Phi)$.  Hence, if $J(\Phi)$ is amenable then $A_n$ is amenable; by the assumed
isomorphism, so is the displayed $E$-sum.  Assertion (a) follows from
Proposition~\ref{prop:CE-necessary} (or Corollary~\ref{cor:lp-not-amenable} in the
case $E=\ell_p(I)$).  Assertion (b) follows from Theorem~\ref{thm:main}(1), applied
to the displayed $E$-sum.
\end{proof}

\begin{remark}\label{rem:Piszczek}
Piszczek obtained sharp criteria for amenability and contractibility of certain
K\"othe co-echelon (DF) sequence algebras; see \cite{Piszczek2019} and the
surveying discussion in \cite{PirkovskiiPiszczek2022}.  Although the settings are
different (locally convex DF-algebras versus Banach algebras), the mechanism in
Corollary~\ref{cor:Esum-detected} is analogous in spirit: amenability imposes strong
uniformity constraints on the underlying sequence structure, and these constraints
can be detected after passing to an appropriate coordinate quotient of a conditional
$J$-sum.
\end{remark}

\section{Weak amenability of unconditional $E$-sums}\label{sec:weak-amenability}

This section records a weak-amenability counterpart of the $\ell_p(A)$/$c_0(A)$ result
highlighted in \cite{KoczorowskiPiszczek2024WeakAmenability}.
We work throughout with arbitrary index sets and unconditional $E$-sums.

\begin{definition}[Weak amenability]\label{def:WA}
Let $A$ be a Banach algebra.
A bounded linear map $D\colon A\to A^*$ is a \emph{derivation} if
\[
D(ab)=a\cdot D(b)+D(a)\cdot b\qquad (a,b\in A),
\]
where the canonical $A$-bimodule structure on $A^*$ is given by
\[
(a\cdot\phi)(x):=\phi(xa),\qquad (\phi\cdot a)(x):=\phi(ax)
\qquad (a,x\in A,\ \phi\in A^*).
\]
For $\phi\in A^*$ the \emph{inner derivation} implemented by $\phi$ is
\[
\ad_\phi(a):=a\cdot\phi-\phi\cdot a\qquad (a\in A).
\]
The Banach algebra $A$ is \emph{weakly amenable} if every bounded derivation $D\colon A\to A^*$ is inner.
\end{definition}

\begin{remark}\label{rem:WA-commutative}
If $A$ is commutative, then $a\cdot\phi=\phi\cdot a$ for all $a\in A$ and $\phi\in A^*$, hence $\ad_\phi=0$ for all $\phi$.
Therefore, for commutative $A$, weak amenability is equivalent to the vanishing of all bounded derivations $A\to A^*$.
\end{remark}

\begin{definition}[Weak amenability constant]\label{def:WAM}
Let $A$ be a Banach algebra. Define
\[
\begin{split}
\WAM(A):=\inf\Bigl\{C\ge 0:\ &\text{for every bounded derivation }D\colon A\to A^*\\
&\text{there exists }\phi\in A^*\text{ such that }D=\ad_\phi
\text{ and }\|\phi\|\le C\|D\|\Bigr\}\in[0,\infty].
\end{split}
\]
with $\WAM(A)=+\infty$ if $A$ is not weakly amenable.
\end{definition}

Let us record the following folk lemma. We provide the proof for the sake of completeness.
\begin{lemma}\label{lem:WA-essential}
If $A$ is weakly amenable, then $\overline{A^2}=A$, where $A^2:=\mathrm{span}\{ab:\ a,b\in A\}$.
\end{lemma}

\begin{proof}
Assume $\overline{A^2}\neq A$. Then by Hahn--Banach there exists $0\neq\psi\in A^*$ with $\psi|_{\overline{A^2}}=0$.
For each $a\in A$, define $D(a):=\psi(a)\,\psi\in A^*$.
Since $\psi(ab)=0$ for all $a,b\in A$, we have $D(ab)=0$.
Moreover, for every $c\in A$,
\[
(a\cdot\psi)(c)=\psi(ca)=0,\qquad (\psi\cdot a)(c)=\psi(ac)=0,
\]
so $a\cdot\psi=\psi\cdot a=0$ for all $a\in A$, and hence
\[
a\cdot D(b)+D(a)\cdot b=\psi(b)\,a\cdot\psi+\psi(a)\,\psi\cdot b=0.
\]
Thus $D$ is a bounded derivation $A\to A^*$.

If $D$ were inner, then for each $a\in A$ we would have $D(a)(a)=0$ because
$\ad_\phi(a)(a)=\phi(aa-aa)=0$.  But choosing $a$ with $\psi(a)\neq 0$ gives
$D(a)(a)=\psi(a)^2\neq 0$, a contradiction. Hence $A$ is not weakly amenable.
\end{proof}

For a Banach sequence lattice $E$ on $I$ and a family $(A_i)_{i\in I}$ of Banach algebras, set
\[
A:=\Bigl(\bigoplus_{i\in I}A_i\Bigr)_{\!E}.
\]
For $i\in I$ we write $\iota_i\colon A_i\to A$ for the coordinate embedding and $\pi_i\colon A\to A_i$ for the coordinate projection.
Note that $\pi_i$ is contractive (by Definition~\ref{def:sequence-lattice}(c)), while $\|\iota_i\|=\|\delta_i\|_E$ in general.

For $F\in\cF(I)$ denote by $A_F$ the closed subalgebra of $A$ supported in $F$ and by $P_F\colon A\to A_F$ the truncation map.

\begin{proposition}\label{prop:WA-to-summands}
If $A=(\oplus_{i\in I}A_i)_E$ is weakly amenable, then each $A_i$ is weakly amenable.

More quantitatively, if $\WAM(A)<\infty$, then for each $i\in I$,
\[
\WAM(A_i)\le \|\iota_i\|\,\WAM(A)=\|\delta_i\|_E\,\WAM(A).
\]
\end{proposition}

\begin{proof}
Fix $i\in I$ and let $d\colon A_i\to A_i^*$ be a bounded derivation.
Define $D\colon A\to A^*$ by
\[
D:=\pi_i^*\circ d\circ \pi_i,
\]
where $\pi_i^*\colon A_i^*\hookrightarrow A^*$ is the adjoint of $\pi_i$.
Since $\pi_i$ is a homomorphism, $\pi_i^*$ is a bimodule map, hence $D$ is a derivation; moreover
$\|D\|\le \|d\|$ because $\pi_i$ is contractive.
If $A$ is weakly amenable then $D=\ad_\Phi$ for some $\Phi\in A^*$.
Set $\phi:=\iota_i^*(\Phi)\in A_i^*$, i.e.\ $\phi(a)=\Phi(\iota_i(a))$.
A routine compatibility check yields $d=\ad_\phi$.
Thus $A_i$ is weakly amenable.

For the quantitative estimate, take any $C>\WAM(A)$ and choose $\Phi$ with
$D=\ad_\Phi$ and $\|\Phi\|\le C\|D\|$.  Then
\[
\|\phi\|\le \|\iota_i\|\,\|\Phi\|\le \|\iota_i\|\,C\,\|d\|.
\]
Taking the infimum over all such $C$ gives
$\WAM(A_i)\le \|\iota_i\|\WAM(A)=\|\delta_i\|_E\WAM(A)$.
\end{proof}

\begin{lemma}[Off-diagonal matrix coefficients]\label{lem:off-diagonal}
Let $A=(\oplus_{i\in I}A_i)_E$ and let $D\colon A\to A^*$ be a bounded derivation.
For $j,k\in I$ define
\[
D_{jk}:=\iota_j^*\circ D\circ\iota_k\colon A_k\longrightarrow A_j^*.
\]
If $j\neq k$ and $\overline{A_j^2}=A_j$, then $D_{jk}=0$.  Consequently, if every $A_j$ is
weakly amenable, then for each $k\in I$ and $a\in A_k$ the functional
$D(\iota_k(a))$ annihilates every coordinate ideal $\iota_j(A_j)$ with $j\neq k$.
\end{lemma}

\begin{proof}
Fix $j\neq k$, $a\in A_k$, and $b,c\in A_j$.  Since
$\iota_k(a)\iota_j(b)=0$, the derivation identity gives
\[
0=D(\iota_k(a)\iota_j(b))
  =\iota_k(a)\cdot D(\iota_j(b))+D(\iota_k(a))\cdot \iota_j(b).
\]
Evaluating at $\iota_j(c)$, the first term is zero because
$\iota_j(c)\iota_k(a)=0$, while the second term is
\[
\bigl(D(\iota_k(a))\cdot\iota_j(b)\bigr)(\iota_j(c))
   =D(\iota_k(a))(\iota_j(bc))=D_{jk}(a)(bc).
\]
Thus $D_{jk}(a)$ annihilates all products $bc$ in $A_j$, and hence annihilates the
linear span $A_j^2$.  If $\overline{A_j^2}=A_j$, continuity of $D_{jk}(a)$ gives
$D_{jk}(a)=0$.  The final assertion follows from Lemma~\ref{lem:WA-essential}.
\end{proof}

\begin{theorem}\label{thm:WA-counterpart}
Let $I$ be a non-empty set, let $E$ be a Banach sequence lattice on $I$, and let $(A_i)_{i\in I}$ be Banach algebras.
Set $A:=(\oplus_{i\in I}A_i)_E$.

\begin{enumerate}[label={\rm(\arabic*)},leftmargin=3.2em]
\item If $A$ is weakly amenable, then each $A_i$ is weakly amenable.
\item If every $A_i$ is commutative and weakly amenable, then $A$ is weakly amenable.
(In fact every derivation $A\to A^*$ is zero.)
\item Assume $I$ is infinite and $E=\ell_p(I)$ for some $1<p<\infty$.
Let $B$ be a weakly amenable \emph{non-commutative} Banach algebra, and form the constant family $A_i=B$.
Then the algebra $(\oplus_{i\in I}B)_{\ell_p(I)}$ is \emph{not} weakly amenable.
\end{enumerate}
\end{theorem}

\begin{proof}
\emph{(1)} This is Proposition~\ref{prop:WA-to-summands}.

\smallskip

\emph{(2)} Let $D\colon A\to A^*$ be a bounded derivation.
Since $A$ is commutative, inner derivations vanish (Remark~\ref{rem:WA-commutative}), so it suffices to prove $D=0$.

By Lemma~\ref{lem:WA-essential}, each summand satisfies $\overline{A_i^2}=A_i$.
Hence Lemma~\ref{lem:off-diagonal} shows that, for $j\in I$ and $a\in A_j$, the
functional $D(\iota_j(a))$ annihilates every coordinate ideal $\iota_i(A_i)$ with
$i\neq j$.

It remains to look at the diagonal coefficient.  Define
\[
d_j:=\iota_j^*\circ D\circ \iota_j\colon A_j\to A_j^*.
\]
Then $d_j$ is a derivation.  Since $A_j$ is commutative and weakly amenable, $d_j=0$.
Thus $D(\iota_j(a))$ also annihilates $\iota_j(A_j)$.  Consequently
$D(\iota_j(a))$ annihilates the algebraic span of all coordinate ideals, which is dense in
$A$ by Lemma~\ref{lem:dense-finite-support}; by continuity of the functional
$D(\iota_j(a))$, it is zero.  We have shown that $D$ vanishes on every coordinate ideal,
hence on the dense subalgebra of finitely supported elements.  Since $D$ is continuous,
$D=0$ on $A$.

\smallskip
\emph{(3)} Let $1<p<\infty$ and put $q=p/(p-1)$.
Choose $\psi\in B^*$ with $\psi\notin \Zc(B^*)$; equivalently $\ad_\psi\neq 0$.
Set $d:=\mathrm{dist}(\psi,\Zc(B^*))>0$.

Choose a countably infinite subset $\{i_n:n\in\mathbb N\}\subset I$.
Define a scalar weight $w=(w_i)_{i\in I}$ by
\[
w_{i_n}:=
\begin{cases}
1, & 1<p\le 2,\\[2pt]
n^{-1/q}, & p>2,
\end{cases}
\qquad\text{and}\qquad
w_i:=0\ \ (i\notin \{i_n:n\in\mathbb N\}).
\]
Then $w\notin \ell_q(I)$, but multiplication by $w$ defines a bounded operator
$M_w\colon \ell_p(I)\to \ell_q(I)$:
for $1<p\le 2$ this is just the continuous embedding $\ell_p\hookrightarrow \ell_q$;
for $p>2$ it follows from H\"older since $w\in \ell_r(I)$ with $r=p/(p-2)$ and $1/q=1/p+1/r$.

Now define $D\colon \ell_p(I,B)\to \ell_q(I,B^*)\cong (\ell_p(I,B))^*$ by
\[
D\bigl((b_i)_{i\in I}\bigr):=\bigl(w_i\,\ad_\psi(b_i)\bigr)_{i\in I}.
\]
Since $\|\ad_\psi(b_i)\|\le 2\|\psi\|\,\|b_i\|$, boundedness of $M_w$ yields that $D$ is a bounded derivation.

Suppose, towards a contradiction, that $D$ is inner: $D=\ad_\Phi$ for some $\Phi\in \ell_q(I,B^*)$.
Restricting to the $i$th coordinate ideal (via the canonical embedding/projection) gives
\[
\ad_{\Phi_i}=w_i\,\ad_\psi\qquad (i\in I),
\]
hence $\Phi_i-w_i\psi\in \Zc(B^*)$ for each $i$.
Therefore
\[
\|\Phi_i\|\ge \mathrm{dist}(w_i\psi,\Zc(B^*))=|w_i|\,\mathrm{dist}(\psi,\Zc(B^*))=d\,|w_i|.
\]
It follows that $(\|\Phi_i\|)_{i\in I}$ dominates $d(|w_i|)_{i\in I}$ and hence cannot lie in $\ell_q(I)$ because $w\notin\ell_q(I)$.
This contradicts $\Phi\in\ell_q(I,B^*)$. Thus $D$ is not inner and $\ell_p(I,B)$ is not weakly amenable.
\end{proof}

\begin{lemma}\label{lem:WAM-c0}
Let $(A_i)_{i\in I}$ be Banach algebras and set $A:=\big(\bigoplus_{i\in I}A_i\big)_{c_0(I)}$.
Then
\[
\WAM(A)=\sup_{i\in I}\WAM(A_i),
\]
with the convention that the right-hand side is $+\infty$ if any $A_i$ is not weakly amenable.
\end{lemma}

\begin{proof}
For each $i$, the coordinate projection $\pi_i\colon A\to A_i$ is contractive and
$\|\iota_i\|=1$.  Proposition~\ref{prop:WA-to-summands} therefore gives
$\WAM(A_i)\le \WAM(A)$, and hence
\[
\sup_{i\in I}\WAM(A_i)\le \WAM(A).
\]

Conversely, assume that $S:=\sup_{i\in I}\WAM(A_i)<\infty$, and let
$D\colon A\to A^*$ be a bounded derivation.  We use the standard identification
\[
A^*=\ell_1\text{-}\bigoplus_{i\in I} A_i^* .
\]
For each $i\in I$ set
\[
d_i:=\iota_i^*\circ D\circ \iota_i\colon A_i\to A_i^* .
\]
Then $d_i$ is a bounded derivation.  Fix a number $C>S$.  By the definition of
$\WAM(A_i)$, choose $\phi_i\in A_i^*$ such that
\[
d_i=\ad_{\phi_i},\qquad \|\phi_i\|\le C\|d_i\|\qquad (i\in I).
\]

We claim that $(\|d_i\|)_{i\in I}\in\ell_1(I)$ and
\[
\sum_{i\in I}\|d_i\|\le \|D\| .
\]
It is enough to prove the corresponding finite estimate.  Let $F\subseteq I$ be a non-empty
finite set and let $\varepsilon>0$.  For each $i\in F$ choose $a_i\in A_i$ with $\|a_i\|\le 1$ and
$\|d_i(a_i)\|\ge \|d_i\|-\varepsilon/|F|$.
Put $a:=\sum_{i\in F}\iota_i(a_i)$.  Then $\|a\|\le 1$.  Since each $A_j$ is weakly
amenable, Lemma~\ref{lem:WA-essential} gives $\overline{A_j^2}=A_j$, and
Lemma~\ref{lem:off-diagonal} shows that $D(\iota_i(a_i))$ is supported on the $i$th
coordinate.  Therefore, using the $\ell_1$-norm in $A^*$,
\[
\|D\|\ge \|D(a)\|=
\sum_{i\in F}\|\iota_i^*D\iota_i(a_i)\|
=\sum_{i\in F}\|d_i(a_i)\|
\ge \sum_{i\in F}\|d_i\|-\varepsilon.
\]
Letting $\varepsilon\downarrow0$ and then taking the supremum over finite $F$ proves the claim.

It follows that $\Phi:=(\phi_i)_{i\in I}$ belongs to $A^*$ and
\[
\|\Phi\|=\sum_{i\in I}\|\phi_i\|\le C\sum_{i\in I}\|d_i\|\le C\|D\|.
\]
For $a\in A_i$, the derivations $D$ and $\ad_\Phi$ agree on $\iota_i(a)$: both have no
off-diagonal coordinates by Lemma~\ref{lem:off-diagonal}, and their $i$th coordinate is
$d_i(a)=\ad_{\phi_i}(a)$.  Hence $D=\ad_\Phi$ on the algebraic span of the coordinate
ideals, which is dense in $A$.  By continuity, $D=\ad_\Phi$ on all of $A$.
Thus $\WAM(A)\le C$.  Since $C>S$ was arbitrary, $\WAM(A)\le S$.
\end{proof}

\begin{corollary}[Weak amenability in the $c_0$-type regime]\label{cor:WAM-CE}
Let $E$ be a Banach sequence lattice on $I$ with $C_E<\infty$.
Let $(A_i)_{i\in I}$ be Banach algebras and set $A=(\oplus_{i\in I}A_i)_E$.
Then $A$ is weakly amenable if and only if $\sup_{i\in I}\WAM(A_i)<\infty$.
Moreover, in that case
\[
C_E^{-1}\sup_{i\in I}\WAM(A_i)\ \le\ \WAM(A)\ \le\ C_E^2\,\sup_{i\in I}\WAM(A_i).
\]
If, in particular, $\|\delta_i\|_E=1$ for all $i\in I$ (as for the usual $c_0(I)$ norm), then the lower
bound improves to $\sup_i\WAM(A_i)\le\WAM(A)$.
For an Orlicz heart $E=h_\varphi(I)$ with $a_\varphi>0$ (so $C_E<\infty$ by Proposition~\ref{prop:orlicz-CE}, after the harmless normalisation from Remark~\ref{rem:orlicz-normalisation}),
weak amenability of $(\oplus_i A_i)_{h_\varphi}$ is equivalent to uniform weak amenability of the summands.
\end{corollary}

\begin{proof}
If $C_E<\infty$, then $E$ coincides with $c_0(I)$ as a set and the two norms are equivalent:
\[
\|x\|_\infty \le \|x\|_E \le C_E\|x\|_\infty \qquad (x\in E),
\]
by Lemma~\ref{lem:CE-equivalence}. Hence the identity map
\[
T\colon (\oplus_{i\in I}A_i)_E \longrightarrow (\oplus_{i\in I}A_i)_{c_0(I)}
\]
is an algebra isomorphism with $\|T\|\le 1$ and $\|T^{-1}\|\le C_E$.
A standard conjugation argument for derivations under algebra isomorphisms gives
\[
\WAM\bigl((\oplus_i A_i)_E\bigr)
\le \|T^{-1}\|^2\,\|T\|^2\,\WAM\bigl((\oplus_i A_i)_{c_0(I)}\bigr)
\le C_E^2\,\WAM\bigl((\oplus_i A_i)_{c_0(I)}\bigr).
\]
By Lemma~\ref{lem:WAM-c0}, the last term is
$C_E^2\sup_{i\in I}\WAM(A_i)$, giving the upper bound and one implication.

Conversely, Proposition~\ref{prop:WA-to-summands} gives
\[
\WAM(A_i)\le \|\delta_i\|_E\,\WAM(A)\le C_E\,\WAM(A)\qquad(i\in I),
\]
which gives the stated lower bound and shows that weak amenability of $A$ forces uniform weak amenability of the summands.
If $\|\delta_i\|_E=1$ for all $i$, the same estimate gives the sharper lower bound.
\end{proof}
\section*{Conflict of Interest and Data Availability}

\paragraph{Conflict of Interest.}
The authors declare that they have no conflict of interest.

\paragraph{Data Availability.}
No datasets were generated or analysed during the current study.


\end{document}